\documentclass[10pt,reqno]{amsart} 

\usepackage{float} 
\usepackage{appendix}
\usepackage{soul}

\usepackage{hhline}
\usepackage[left=2cm, right=2cm, top=2.5cm, bottom=2.5cm]{geometry}
\usepackage{epsfig,amssymb,amsmath,version}
\usepackage{amssymb,version,graphicx,fancybox,mathrsfs}

\usepackage{enumitem}
 \usepackage{bbm}
\usepackage{url,hyperref}
\usepackage{subfigure}
\usepackage{color}
\usepackage{stmaryrd}
\usepackage{multirow}
\usepackage{booktabs,siunitx}
\usepackage{booktabs}
\usepackage{multicol,epstopdf}
\usepackage{xypic}
\usepackage[]{graphicx}
\usepackage[all]{xy}
\usepackage{tikz,ifthen}
\usetikzlibrary{positioning, math, decorations.markings, arrows.meta}
\usetikzlibrary{calc,shapes,cd}
\usetikzlibrary{knots} 
\usepgfmodule{decorations}
\allowdisplaybreaks
\usepackage[
backend=biber,
style=alphabetic,
sorting=nyt
]{biblatex}
\renewbibmacro*{in:}{%
  \ifentrytype{article}
    {}
    {\printtext{\bibstring{in}\intitlepunct}}%
}

\DeclareFieldFormat[article]{title}{#1}

\tikzset{knotarrow/.pic={ \draw[edge, <-] (0,0) -- +(-.001,0);}}

\tikzset{edge/.style={line width=0.8}}
\tikzset{wall/.style={very thick}}

\tikzset{->-/.style n args={2}{decoration={markings, mark=at position #1 with {\arrow{#2}}}, postaction={decorate}}}

\tikzset{-o-/.code 2 args={\ifstreqF{#2}{} 
{\ifstreqTF{#2}{>}
   {\pgfkeysalso{decoration={markings,mark=at position #1 with {\arrow[scale=0.8]{#2}}}
                    ,postaction={decorate}}
    }
   {\ifstreqTF{#2}{<}
       {\pgfkeysalso{decoration={markings,mark=at position #1 with {\arrow[scale=0.8]{#2}}}
                    ,postaction={decorate}}
        }
       {\pgfkeysalso{decoration={markings,
                    mark=at position #1 with
                    {\draw[black, fill={#2}] circle[radius=2pt];}}
                    ,postaction={decorate}}
        }
     }
  }}}

\DeclareFieldFormat[article]{title}{#1}

\allowdisplaybreaks[4]

\def\dS{\widetilde{\cS}_\omega(\fS)}

\newtheorem{thm}{Theorem}[section]
\newtheorem{theorem}{Theorem}[section]
\newtheorem{lemma}[theorem]{Lemma}
\newtheorem{definition}[theorem]{Definition}
\newtheorem{corollary}[theorem]{Corollary}
\newtheorem{proposition}[theorem]{Proposition}
\newtheorem{example}[theorem]{Example}
\newtheorem{remark}[theorem]{Remark}
\newtheorem{conjecture}[theorem]{Conjecture}

\def\cS{\mathscr S}

\newcommand{\wideoverunder}[2]{ 
\!\begin{array}{c} 
\scriptstyle{#1}\\[-.1in] 
-\!\!\!-\!\!\!-\!\!\!-\!\!\!-\!\!\!-\\[-.1in] 
\scriptstyle{#2} 
\end{array} 
\! 
}

\newcommand{\bp}{\begin{proposition}}
\newcommand{\ep}{\end{proposition}}
\newcommand{\bpr}{\begin{proof}}
\newcommand{\epr}{\end{proof}}

\newcommand{\bt}{\begin{theorem}}
\newcommand{\et}{\end{theorem}}

\newcommand{\bl}{\begin{lemma}}
\newcommand{\el}{\end{lemma}}

\newcommand{\bcr}{\begin{corollary}}
\newcommand{\ecr}{\end{corollary}}

\newcommand{\be}{\begin{equation}}
\newcommand{\ee}{\end{equation}}

\newcommand{\bes}{\begin{equation*}}
\newcommand{\ees}{\end{equation*}}

\newcommand{\ba}{\begin{align}}
\newcommand{\ea}{\end{align}}

\newcommand{\bas}{\begin{align*}}
\newcommand{\eas}{\end{align*}}

\newcommand{\gaa}{\mathsf{g}}

\DeclareMathOperator{\skeleton}{\textbf{sk}}
\DeclareMathOperator{\im}{\mathrm{Im}}

\DeclareMathOperator{\Int}{\mathrm{Int}}

\DeclareMathOperator{\sgn}{\mathrm{sgn}}

\def\fS{\mathfrak{S}}

\graphicspath{{./Figures/} }

\newcommand{\blue}[1]{{\color{blue}#1}}

\renewbibmacro{in}{}
\begin{document}

\title{A mutation invariant for skew-symmetrizable matrices}

\author[Min Huang]{Min Huang}
\address{Min Huang, School of Mathematics (Zhuhai), Sun Yat-sen University, Zhuhai, China.}
\email{huangm97@mail.sysu.edu.cn}

\author[Qiling Ma]{Qiling Ma}
\address{Qiling Ma, School of Mathematics (Zhuhai), Sun Yat-sen University, Zhuhai, China.}
\email{maqling@mail2.sysu.edu.cn}


 \maketitle

\begin{abstract}
Matrix mutation of skew-symmetrizable matrices is foundational in cluster algebra theory.  
Effective mutation invariants are essential for determining whether two matrices lie in the same mutation class.  
Casals~\cite{Casals} introduced a binary mutation invariant for skew-symmetric matrices.  
In this paper, we extend Casals' construction to the skew-symmetrizable setting.  
When the skew-symmetrizer $d_1,\dots, d_n$ is pairwise coprime, we obtain two distinct extensions of this invariant.
\end{abstract}

\newcommand{\ca}{{\cev{a}  }}

\def\BZ{\mathbb Z}
\def\Id{\mathrm{Id}}
\def\Mat{\mathrm{Mat}}
\def\BN{\mathbb N}

\def \cb {\color{blue}}
\def \cred {\color{red}}
\def \cbf {\color{blue}\bf}
\def \credf {\color{red}\bf}
\definecolor{ligreen}{rgb}{0.0, 0.3, 0.0}
\def \cg {\color{ligreen}}
\def \cgf {\color{ligreen}\bf}
\definecolor{darkblue}{rgb}{0.0, 0.0, 0.55}
\def \dbf {\color{darkblue}\bf}
\definecolor{anti-flashwhite}{rgb}{0.55, 0.57, 0.68}
\def \afw {\color{anti-flashwhite}}
\def\cF{\mathbb F}
\def\Z{\mathbb Z}
\def\cP{\mathcal P}
\def\embed{\hookrightarrow}
\def\pr{\mathrm{pr}}
\def\cV{\mathcal V}
\def\ot{\otimes}
\def\buu{{\mathbf u}}


\def \ri {{\rm i}}
\newcommand{\bs}[1]{\boldsymbol{#1}}
\newcommand{\cev}[1]{\reflectbox{\ensuremath{\vec{\reflectbox{\ensuremath{#1}}}}}}
\def\bS{\bar \fS}
\def\cE{\mathcal E}
\def\fB{\mathfrak B}
\def\cR{\mathcal R}
\def\cY{\mathcal Y}
\def\cS{\mathscr S}
\def\rS{\overline{\cS}_\omega}

\def\fS{\mathfrak{S}}

\def\MN {(M)}
\def\cN {\mathcal{N}}
\def\SL{{\rm SL}_n}

\def\bP{\mathbb P}
\def\bR{\mathbb R}

\def\SS{\cS_{\omega}(\fS)}
\def\rdS{\overline \cS_{\omega}(\fS)}
\def\rdP{\overline \cS_{\omega}(\mathbb{P}_4)}

\def\Vm{\mathcal V_{\text{mut}}}
\def\Vc{\mathcal V}

\newcommand{\beq}{\begin{equation}}
	\newcommand{\eeq}{\end{equation}}

\section{Introduction and Main result}\label{sec-intro}

Matrix mutation for skew-symmetrizable integer matrices was introduced by Fomin and Zelevinsky~\cite{FZ} and is foundational in cluster algebra theory.  
Two skew-symmetrizable integer matrices are said to be \emph{mutation equivalent} if one can be obtained from the other by a finite sequence of mutations, possibly followed by a simultaneous permutation of rows and columns.  

Determining whether two given $n \times n$ skew-symmetrizable matrices are mutation equivalent is an important open problem; see e.g.~\cite{F},\cite[Problem~2.8.2]{FWZ}.  
A standard approach to this problem is to construct \emph{mutation invariants}—quantities that remain unchanged under mutation. Beineke, Br\"ustle, and Hille~\cite{BBH} introduced the \emph{Markov invariant} for $3 \times 3$ skew-symmetric matrices, which essentially characterizes those that are mutation-acyclic.  
Casals~\cite{Casals} used the determinant of the symmetrized matrix to define a binary mutation invariant, referred to as the $\delta$-invariant.  
Fomin and Neville~\cite{FN} introduced the notion of a \emph{totally proper quiver} (further studied by Neville~\cite{N}), which yields powerful mutation invariants.
For additional examples of mutation invariants, we refer the reader to \cite{ABBS,E,E2,FN1,M,S1,S2,S3}.  
However, very few mutation invariants are known for general skew-symmetrizable matrices. 

Let $M_n(\mathbb{Z})$ denote the set of $n \times n$ integer matrices.  
A matrix $B = (b_{ij}) \in M_n(\mathbb{Z})$ is called \emph{skew-symmetrizable} if there exists a diagonal matrix 
$D = \operatorname{diag}(d_1, \dots, d_n)$ with positive integers $d_i$ such that $DB$ is skew-symmetric.  
Such a matrix $D$ is called a \emph{skew-symmetrizer} of $B$.

Following~\cite[Section~2.7]{FWZ}, fix an index $k \in [n] := \{1, \dots, n\}$, and define three matrices $J_{k,n}, E_{k,n}, F_{k,n} \in M_n(\mathbb{Z})$ as follows:
\begin{itemize}
    \item $J_{k,n}$ is the $n \times n$ diagonal matrix whose diagonal entries are all $1$, except for the $(k,k)$-entry, which equals $-1$;
    
    \item $E_{k,n} = (e_{ij})$ is the $n \times n$ matrix given by
    \[
        e_{ik} = \max(0, -b_{ik}) \quad \text{for all } i \in [n], \qquad e_{ij} = 0 \text{ if } j \ne k.
    \]
    \item $F_{k,n} = (f_{ij})$ is the $n \times n$ matrix given by
    \[
        f_{ik} = \max(0, b_{ki}) \quad \text{for all } i \in [n], \qquad f_{ij} = 0 \text{ if } i \ne k.
    \]
\end{itemize}
Set $M_k := J_{k,n} + E_{k,n}, N_k:= J_{k,n}+F_{k,n}$; note that $M_k, N_k$ depends on both $k$ and $B$.  
The \emph{mutation} of $B$ at $k$ is then defined as
\[
    \mu_k(B) := M_k \, B \, N_k.
\]
Note that $N_k=M_k^{\!\top}$ in case $B$ is skew-symmetric.

It is well known that $\mu_k(B)$ is again skew-symmetrizable, and in fact admits the same skew-symmetrizer $D$ as $B$. Consequently, mutation can be iterated. Moreover, mutation is involutive: $\mu_k(\mu_k(B)) = B$ for all $k \in [n]$. 

We adopt the following definition from~\cite[Definition~2.8.1]{FWZ}:

\begin{definition}\label{def:mutation}
Two skew-symmetrizable matrices $B, B' \in M_n(\mathbb{Z})$ are \emph{mutation equivalent} if one can be obtained from the other by a finite sequence of mutations, possibly followed by a simultaneous permutation of rows and columns.
\end{definition}

Let $B = (b_{ij}) \in M_n(\mathbb{Z})$ be a skew-symmetrizable integer matrix.  
Following~\cite[Definition~2.7.8]{FWZ}, we define its \emph{symmetrization} to be the symmetric matrix 
$\mathfrak{S}(B) = (s_{ij}) \in M_n(\mathbb{R})$ given by
\begin{equation*}
    s_{ij} =
    \begin{cases}
        2, & \text{if } i = j,\\
        \operatorname{sgn}(b_{ij}) \sqrt{|b_{ij} b_{ji}|}, & \text{if } i < j, \\
        \operatorname{sgn}(b_{ji}) \sqrt{|b_{ij} b_{ji}|}, & \text{if } i > j.
    \end{cases}
\end{equation*}
where $\operatorname{sgn}(x)$ denotes the sign of $x$, with the convention $\operatorname{sgn}(0) = 0$.

\begin{lemma}\label{lem:integer}
Let $B \in M_n(\mathbb{Z})$ be skew-symmetrizable. Then $\det(\mathfrak{S}(B)) \in \mathbb{Z}$.
\end{lemma}

We prove Lemma \ref{lem:integer} in Section \ref{subsec:proof1}.

We generalize Casals' $\delta$-invariant to all skew-symmetrizable integer matrices as follows.

\begin{definition}
The \emph{$\delta$-invariant} of a skew-symmetrizable matrix $B \in M_n(\mathbb{Z})$ is defined as
\[
    \delta(B) := \det(\mathfrak{S}(B)) \pmod 4,
\]
i.e., the reduction of $\det(\mathfrak{S}(B))$ modulo $4$.
\end{definition}

We conjecture that it is also binary for any fixed positive integer $n$ and any fixed skew-symmetrizer $d_1, \dots, d_n$. See Remark~ \ref{rem:binary}.

\begin{lemma}\label{lem:permutation}
  Let $B\in M_n(\Z)$ be an $n\times n$ skew-symmetrizable matrix. Then, for any permutation matrix $P\in M_n(\Z)$, we have 
  $$\delta(PBP^{\!\top})\equiv \delta(B) \pmod{4}.$$
\end{lemma}

We prove Lemma \ref{lem:permutation} in Section \ref{subsec:proof2}.

The main result of this article is the following:

\begin{thm}\label{thm:main}
Let $B \in M_n(\mathbb{Z})$ be a skew-symmetric integer matrix. If $B' \in M_n(\mathbb{Z})$ is mutation equivalent to $B$, then
\[
    \delta(B') \equiv \delta(B) \pmod{4}.
\]
\end{thm}

We prove Theorem \ref{thm:main} in Section \ref{subsec:proof3}.

\begin{remark}\label{rem:alt-sym}
Suppose the skew-symmetrizer $d_1, \dots, d_n$ satisfies $\gcd(d_i, d_j)=1$ for all $i \neq j$.  
One may then define an alternative symmetrization $\mathfrak{S}'(B)=(s'_{ij})$ by
  \begin{equation*}
    s'_{ij} =
    \begin{cases}
        2d_i, & \text{if } i = j,\\
        \operatorname{sgn}(b_{ij}) \sqrt{|b_{ij} b_{ji}|}, & \text{if } i < j, \\
        \operatorname{sgn}(b_{ji}) \sqrt{|b_{ij} b_{ji}|}, & \text{if } i > j.
    \end{cases} 
  \end{equation*}
In this case, one has $\det(\mathfrak{S}'(B))\in (\prod_{i=1}^nd_i)\Z$, and the quantity $\delta'(B):= \det(\mathfrak{S}'(B)) \pmod {4\prod_{i=1}^nd_i}$ defines another mutation invariant. Since the proof is similar to that of Theorem~\ref{thm:main}, we omit it.
\end{remark}

\begin{example}
Let  $ x_1, x_2, x_3 \in \mathbb{Z} $ , and consider the skew-symmetrizable matrix
$$
B(x_1,x_2,x_3) =
\begin{pmatrix}
0 & 2x_1 & 3x_2 \\
-x_1 & 0 & 3x_3 \\
-x_2 & -2x_3 & 0 
\end{pmatrix},
$$
which admits the skew-symmetrizer  $ (d_1,d_2,d_3) = (1,2,3) $ .

The standard symmetrization $\mathfrak{S}(B)$ and alternative symmetrization $\mathfrak{S}'(B)$ (as in Remark~\ref{rem:alt-sym}) are given by
$$
\mathfrak{S}(B) =
\begin{pmatrix}
2 & \sqrt{2}\,x_1 & \sqrt{3}\,x_2 \\
\sqrt{2}\,x_1 & 2 & \sqrt{6}\,x_3 \\
\sqrt{3}\,x_2 & \sqrt{6}\,x_3 & 2 
\end{pmatrix}, \qquad
\mathfrak{S}'(B) =
\begin{pmatrix}
2 & \sqrt{2}\,x_1 & \sqrt{3}\,x_2 \\
\sqrt{2}\,x_1 & 4 & \sqrt{6}\,x_3 \\
\sqrt{3}\,x_2 & \sqrt{6}\,x_3 & 6 
\end{pmatrix},
$$
and one computes
$$
\delta(B(x_1,x_2,x_3)) \equiv 2x_2^2 \pmod{4}, \qquad \delta'(B(x_1,x_2,x_3)) \equiv 12\bigl(x_1^2 + x_2^2 + x_3^2 + x_1x_2x_3\bigr) \pmod{24}.
$$

In particular, we have
$$
\delta(B(1,0,1))=\delta(B(1,0,0)) \equiv 0 \pmod{4},\quad \delta(B(1,1,0))\neq \delta(B(1,0,0))
$$
but
$$
\delta'(B(1,0,1))\neq  \delta'(B(1,0,0)), \quad \delta(B(1,1,0))=\delta(B(1,0,1)) \equiv 0 \pmod{24}.
$$

Thus, the two invariants $\delta$ and $\delta'$ can distinguish matrices that either invariant alone cannot.
\end{example}

\section{Proof of the Main Result}\label{sec:proof}

In this section, we prove Lemmas~\ref{lem:integer}, \ref{lem:permutation} and Theorem~\ref{thm:main}.

\subsection{Proof of Lemma~\ref{lem:integer} and a sufficient condition for integrality of the determinant}\label{subsec:proof1}

For convenience, we adopt the convention that $b_{ii} = 2$ for all $i \in [n]$.  
Under this convention, the entries of the symmetrized matrix $\mathfrak{S}(B) = (s_{ij})$ are given by
\begin{equation*}
    s_{ij} =
    \begin{cases}
        \operatorname{sgn}(b_{ij}) \sqrt{|b_{ij} b_{ji}|}, & \text{if } i \leq j, \\
        \operatorname{sgn}(b_{ji}) \sqrt{|b_{ij} b_{ji}|}, & \text{if } i > j.
    \end{cases}
\end{equation*}

The following lemma records a well-known property of skew-symmetrizable matrices.

\begin{lemma}\label{lem:sym}
  Let $B = (b_{ij}) \in M_n(\mathbb{Z})$ be a skew-symmetrizable matrix, and let $D = \operatorname{diag}(d_1, \dots, d_n)$ with $d_i \in \mathbb{Z}_{>0}$ be a skew-symmetrizer, so that $DB$ is skew-symmetric.  
  Then for any permutation $\tau \in S_n$, we have
  \begin{equation*}
    \bigl| b_{1\,\tau(1)}\, b_{2\,\tau(2)} \cdots b_{n\,\tau(n)} \bigr|
    =
    \bigl| b_{\tau(1)\,1}\, b_{\tau(2)\,2} \cdots b_{\tau(n)\,n} \bigr|,
  \end{equation*}
  where we use the convention $b_{ii} = 2$ for all $i$.
\end{lemma}

\begin{proof}
Since $DB$ is skew-symmetric, we have $d_i b_{ij} = -d_j b_{ji}$ for all $i,j$, and hence
\[
|d_i b_{ij}| = |d_j b_{ji}|.
\]
This identity also holds when $i = j$: under our convention $b_{ii} = 2$, both sides equal $2d_i$.

Now consider the product
\[
\prod_{i=1}^n |d_i b_{i\,\tau(i)}|
= \prod_{i=1}^n |d_{\tau(i)} b_{\tau(i)\,i}|
= \prod_{j=1}^n |d_j b_{j\,\tau^{-1}(j)}|,
\]
where we made the change of index $j = \tau(i)$. Since $\tau$ is a permutation, the sets $\{\tau(1),\dots,\tau(n)\}$ and $\{1,\dots,n\}$ coincide, so
\[
\Bigl( \prod_{i=1}^n d_i \Bigr) \cdot \bigl| b_{1\,\tau(1)} \cdots b_{n\,\tau(n)} \bigr|
=
\Bigl( \prod_{i=1}^n d_i \Bigr) \cdot \bigl| b_{\tau(1)\,1} \cdots b_{\tau(n)\,n} \bigr|.
\]
Because each $d_i$ is a positive integer, the common factor $\prod_{i=1}^n d_i > 0$ can be cancelled, yielding
\begin{equation*}
    \bigl| b_{1\,\tau(1)} \cdots b_{n\,\tau(n)} \bigr|
    =
    \bigl| b_{\tau(1)\,1} \cdots b_{\tau(n)\,n} \bigr|,
\end{equation*}
as claimed.
\end{proof}

\medskip

\begin{proof}[{Proof of Lemma~\ref{lem:integer}}]
For each $\tau \in S_n$, we have
\[
\begin{aligned}
|s_{1\,\tau(1)} s_{2\,\tau(2)} \cdots s_{n\,\tau(n)}|
&= \Bigl| \sqrt{|b_{1\,\tau(1)} b_{\tau(1)\,1}|} \,
          \sqrt{|b_{2\,\tau(2)} b_{\tau(2)\,2}|} \cdots
          \sqrt{|b_{n\,\tau(n)} b_{\tau(n)\,n}|} \Bigr| \\
&= \sqrt{ \bigl| b_{1\,\tau(1)} \cdots b_{n\,\tau(n)} \bigr| \cdot
         \bigl| b_{\tau(1)\,1} \cdots b_{\tau(n)\,n} \bigr| }.
\end{aligned}
\]
By Lemma~\ref{lem:sym}, the two products inside the absolute values are equal in magnitude; hence
\begin{equation}\label{eq:integer}
|s_{1\,\tau(1)} \cdots s_{n\,\tau(n)}|
= \bigl| b_{1\,\tau(1)} \cdots b_{n\,\tau(n)} \bigr| \in \mathbb{Z}.
\end{equation}

Consequently, every term in the expansion of $\det(\mathfrak{S}(B))$ is an integer, and therefore
\[
\det(\mathfrak{S}(B)) = \sum_{\tau \in S_n} (-1)^{\ell(\tau)} \,
s_{1\,\tau(1)} s_{2\,\tau(2)} \cdots s_{n\,\tau(n)} \in \mathbb{Z}.
\]
This completes the proof.
\end{proof}

\medskip

As a consequence of the argument in Lemma~\ref{lem:sym}, we obtain the following general sufficient condition for a real symmetric matrix to have an integer determinant—a fact that will be useful in the sequel.

\begin{lemma}\label{lem:integerdet}
  Let $S = (s_{ij}) \in M_n(\mathbb{R})$ be a symmetric matrix.  
  Suppose that for every pair $(i,j)$ with $i \leq j$, the entry $s_{ij}$ can be expressed as
  \[s_{ij} = \sum_{t \in T_{ij}}\epsilon(t)\sqrt{\,|b^t_{ij} b^t_{ji}|\,},\]
  where $T_{ij}$ is a finite index set and $\epsilon(t), b^t_{ij}, b^t_{ji} \in \mathbb{Z}$ for all $t \in T_{ij}$.  
  Moreover, assume there exist positive integers $d_1, \dots, d_n \in \mathbb{Z}_{>0}$ (called skew-symmetrizer in the sequel) such that
    $d_i b^t_{ij} = -\, d_j b^t_{ji}$
  for all $i, j \in [n]$ and all $t \in T_{ij}$.  

  Then $\det(S) \in \mathbb{Z}$. Furthermore, for all $i, j \in [n]$ and all $t \in T_{ij}$,
  \[
    \sqrt{\,|b^t_{ij} b^t_{ji}|\,}\, \det(S_{i,j}) \in \mathbb{Z},
  \]
  where $S_{i,j}$ denotes the $(n-1) \times (n-1)$ submatrix of $S$ obtained by deleting the $i$-th row and $j$-th column.
\end{lemma}

Since the proof of Lemma~\ref{lem:integerdet} is identical to that of Lemma~\ref{lem:integer}, we omit it.

\begin{remark}\label{rem:subinteger}
  Under the hypotheses of Lemma~\ref{lem:integerdet}, let $S' = (s'_{ij})$ be any symmetric matrix obtained by choosing, for each pair $(i,j)$ with $i \leq j$, a non-empty subset $T'_{ij} \subseteq T_{ij}$ and setting
  \[
    s'_{ij} = \sum_{t \in T'_{ij}} \epsilon(t) \sqrt{\,|b^t_{ij} b^t_{ji}|\,}.
  \]
  Then $S'$ also satisfies the hypotheses of Lemma~\ref{lem:integerdet}, and consequently $\det(S') \in \mathbb{Z}$.
\end{remark}

\subsection{A Lemma on determinate of matrices}

%

For any $i,j \in [n]$, let $\varepsilon_{ij}$ denote the elementary matrix whose entries are all zero except for the $(i,j)$-entry, which equals $1$.

\begin{lemma}\label{lem:det4}
  Let $S = (s_{ij}) \in M_n(\mathbb{R})$ be a real symmetric matrix satisfying the hypotheses of Lemma~\ref{lem:integerdet}, with a fixed skew-symmetrizer $d_1, \dots, d_n \in \mathbb{Z}_{>0}$.  
  Let $i, j \in [n]$ be distinct indices, and let $a =\epsilon \sqrt{|b_{ij} b_{ji}|}$ for some integers $\epsilon, b_{ij}, b_{ji}$ such that
    $d_i b_{ij} = -\, d_j b_{ji}$.
  Then
  \[
    \det\bigl( S - 2a (\varepsilon_{ij} + \varepsilon_{ji}) \bigr)
    \equiv \det(S) \pmod{4}.
  \]
\end{lemma}

\begin{proof}
We expand the determinant of the rank-two perturbation $S - 2a(\varepsilon_{ij} + \varepsilon_{ji})$ using the Laplace expansion along rows $i$ and $j$. This yields
\begin{equation}\label{eq:det1}
\det\bigl( S - 2a (\varepsilon_{ij} + \varepsilon_{ji}) \bigr)= \det(S)
   - (-1)^{i+j} 2a \bigl( \det(S_{i,j}) + \det(S_{j,i}) \bigr) - 4a^2 \det\bigl( S_{\{i,j\},\{i,j\}} \bigr),
\end{equation}
where $S_{i,j}$ denotes the submatrix obtained by deleting row $i$ and column $j$, and $S_{\{i,j\},\{i,j\}}$ is the $(n-2)\times(n-2)$ principal submatrix obtained by removing rows and columns $i$ and $j$.

Since $S$ is symmetric, we have $S_{j,i} = (S_{i,j})^\top$, and therefore
\[
\det(S_{i,j}) = \det(S_{j,i}).
\]

By Lemma~\ref{lem:integerdet}, the product $a\, \det(S_{i,j})$ is an integer. Consequently,
\begin{equation}\label{eq:det2}
(-1)^{i+j} 2a \bigl( \det(S_{i,j}) + \det(S_{j,i}) \bigr)
= (-1)^{i+j} 4a \det(S_{i,j}) \equiv 0 \pmod{4}.
\end{equation}

Moreover, the principal submatrix $S_{\{i,j\},\{i,j\}}$ satisfies the hypotheses of Lemma~\ref{lem:integerdet} (with the symmetrizing weights $\{d_k\}_{k \ne i,j}$), so by Lemma~\ref{lem:integerdet}, we have $\det\bigl(S_{\{i,j\},\{i,j\}}\bigr) \in \mathbb{Z}$. Since $a^2 =\epsilon^2 |b_{ij} b_{ji}| \in \mathbb{Z}$, it follows that
\begin{equation}\label{eq:det3}
4a^2 \det\bigl( S_{\{i,j\},\{i,j\}} \bigr) \equiv 0 \pmod{4}.
\end{equation}

Combining \eqref{eq:det1}, \eqref{eq:det2}, and \eqref{eq:det3}, we conclude that
\[
\det\bigl( S - 2a (\varepsilon_{ij} + \varepsilon_{ji}) \bigr)
\equiv \det(S) \pmod{4},
\]
as claimed.
\end{proof}


\subsection{Proof of Lemma \ref{lem:permutation}}\label{subsec:proof2}

The following lemma is an analogue of \cite[Lemma~2.5]{Casals}; it can be verified by a straightforward computation.

\begin{lemma}\label{lem:permutation2}
  Let $P \in M_n(\mathbb{Z})$ be the permutation matrix corresponding to the simple transposition $s_k$ for some $k \in [n-1]$.  
  Then for any skew-symmetrizable matrix $B \in M_n(\mathbb{Z})$, we have
  \[
    P\,\mathfrak{S}(B)\,P^{\!\top} - \mathfrak{S}(P B P^{\!\top})
    = 2\,\operatorname{sgn}(b_{k,k+1}) \sqrt{|b_{k,k+1} b_{k+1,k}|}\,(\varepsilon_{k,k+1} + \varepsilon_{k+1,k}).
  \]
\end{lemma}

\begin{proof}[Proof of Lemma~\ref{lem:permutation}]
  Since every permutation can be written as a product of adjacent (simple) transpositions, it suffices to verify the congruence
  \[
    \delta(B) \equiv \delta(P B P^{\!\top}) \pmod{4}
  \]
  in the case where $P$ is the permutation matrix corresponding to the simple transposition $s_k = (k\; k+1)$ for some $k \in [n-1]$.


  Fix a skew-symmetrizer $d_1, \dots, d_n \in \mathbb{Z}_{>0}$ for $B$.  
  The matrix $S := \mathfrak{S}(P B P^{\!\top})$ satisfies the hypotheses of Lemma~\ref{lem:integerdet} with the permuted skew-symmetrizer
  $(d_1, \dots, d_{k-1}, d_{k+1}, d_k, d_{k+2}, \dots, d_n)$.

  Set 
  \[
    a := -\,\operatorname{sgn}(b_{k,k+1}) \sqrt{|b_{k,k+1} b_{k+1,k}|}.
  \]
  The relation $d_{k+1} b_{k+1,k} = -d_k b_{k,k+1}$ ensures that the pair $(S, a)$ satisfies the assumptions of Lemma~\ref{lem:det4}.  
  Consequently,
  \begin{equation}\label{eq:detmod1}
    \det\Bigl( S + 2\,\operatorname{sgn}(b_{k,k+1}) \sqrt{|b_{k,k+1} b_{k+1,k}|}\,(\varepsilon_{k,k+1} + \varepsilon_{k+1,k}) \Bigr)
    \equiv \det(S) \pmod{4}.
  \end{equation}

  Now observe that
  \[
  \begin{aligned}
    \delta(B)
    &= \det(\mathfrak{S}(B))
     = \det(P\, \mathfrak{S}(B)\, P^{\!\top}) \\
    &\overset{\text{Lemma~\ref{lem:permutation2}}}{=}
      \det\Bigl( \mathfrak{S}(P B P^{\!\top}) + 2\,\operatorname{sgn}(b_{k,k+1}) \sqrt{|b_{k,k+1} b_{k+1,k}|}\,(\varepsilon_{k,k+1} + \varepsilon_{k+1,k}) \Bigr) \\
    &\overset{\eqref{eq:detmod1}}{\equiv}
      \det\bigl( \mathfrak{S}(P B P^{\!\top}) \bigr) \pmod{4} \\
    &= \delta(P B P^{\!\top}),
  \end{aligned}
  \]
  as required.
\end{proof}

\subsection{Proof of Theorem \ref{thm:main}}\label{subsec:proof3}

\begin{proof}
By Lemma~\ref{lem:permutation}, it suffices to consider the case where $B' = \mu_k(B) = (b'_{ij})$. In this case, we aim to prove
\begin{equation}\label{eq:detmod2}
  \det(\mathfrak{S}(\mu_k(B))) \equiv \det(\mathfrak{S}(B)) \pmod{4}.
\end{equation}

Applying a simultaneous permutation of rows and columns (which preserves the determinant modulo~4 by Lemma~\ref{lem:permutation}), we may assume without loss of generality that
\[
b_{ik} \geq 0 \quad \text{for all } i \leq k, \qquad 
b_{ik} < 0 \quad \text{for all } i > k.
\]
Denote $\mathfrak{S}(B) = (s_{ij})$ and $\mathfrak{S}(\mu_k(B)) = (s'_{ij})$. Under this assumption, we have $s_{ik} = \sqrt{|b_{ik} b_{ki}|}$ for all $i$.

The entries of the mutated matrix $B' = \mu_k(B)$ are given by
\[
b'_{ij} =
\begin{cases}
 -b_{ij}, & \text{if } i = k \text{ or } j = k, \\
 b_{ij} + [-b_{ik}]_+ b_{kj} + b_{ik} [b_{kj}]_+, & \text{otherwise},
\end{cases}
\]
but under our sign convention (all $b_{ik} \geq 0$ for $i \leq k$, $<0$ for $i > k$), this simplifies to
\[
b'_{ij} =
\begin{cases}
 -b_{ij}, & \text{if } i = k \text{ or } j = k, \\
 b_{ij} + b_{ik} b_{kj}, & \text{if } i < k < j, \\
 b_{ij} - b_{ik} b_{kj}, & \text{if } j < k < i, \\
 b_{ij}, & \text{otherwise}.
\end{cases}
\]

In particular, when $i < k < j$, a direct computation yields
\begin{align*}
(s'_{ij})^2 &= -b'_{ij} b'_{ji} \\
&= -(b_{ij} + b_{ik} b_{kj})(b_{ji} - b_{jk} b_{ki}) \\
&= -b_{ij} b_{ji} + b_{ij} b_{jk} b_{ki} - b_{ik} b_{kj} b_{ji} + b_{ik} b_{ki} b_{jk} b_{kj} \\
&= s_{ij}^2 + 2 b_{ij} b_{jk} b_{ki} + s_{ik}^2 s_{jk}^2 \\
&= \bigl(s_{ij} + s_{ik} s_{jk}\bigr)^2.
\end{align*}

Therefore, the entries of $\mathfrak{S}(B') = (s'_{ij})$ satisfy
\begin{equation}\label{eq:s'}
s'_{ij} =
\begin{cases}
 -s_{ij}, & \text{if } i = k < j \text{ or } i < k = j, \\
 s_{ij} + s_{ik} s_{jk} \text{ or } -s_{ij} - s_{ik} s_{jk}, & \text{if } i < k < j, \\
 s_{ij}, & \text{otherwise}.
\end{cases}
\end{equation}

Now define the matrices
\[
J = I_n - 2\varepsilon_{k,k}, \qquad
E = (e_{ij})_{1 \le i,j \le n},
\]
where $I_n$ is the $n \times n$ identity matrix, and $E$ is given by
\[
e_{ij} =
\begin{cases}
 -s_{ik}, & \text{if } j = k \text{ and } i > k, \\
 0, & \text{otherwise}.
\end{cases}
\]

Set
\[
S'' := (J + E)\, \mathfrak{S}(B)\, (J + E)^\top = (s''_{ij}).
\]
A straightforward computation shows that $S''$ is symmetric and that its entries satisfy
\begin{equation}\label{eq:s''}
s''_{ij} =
\begin{cases}
 -s_{ij}, & \text{if } i < k = j, \\
 s_{ij} - s_{ik} s_{jk}, & \text{if } i < k < j, \\
 s_{ij}, & \text{otherwise}.
\end{cases}
\end{equation}

Let $T = S'' - \mathfrak{S}(B') = (t_{ij})$. Then $T$ is symmetric, and from \eqref{eq:s'} and \eqref{eq:s''}, we obtain for all $i \leq j$:
\[
t_{ij} = s''_{ij} - s'_{ij} =
\begin{cases}
 2s_{ij}, & \text{if } i=k<j, \\
 -2 s_{ik} s_{kj} \text{ or } 2s_{ij}, & \text{if } i < k < j, \\
 0, & \text{otherwise}.
\end{cases}
\]

Fix a skew-symmetrizer $d_1, \dots, d_n \in \mathbb{Z}_{>0}$ for $B$. This same tuple also serves as a skew-symmetrizer for $B'$, since mutation preserves the skew-symmetrizable property with the same weights.

We now verify that each nonzero entry $t_{ij}, (i\leq j)$ satisfies the hypotheses of Lemma~\ref{lem:det4}.  

- If $t_{ij} = 2s_{ij}$, then $s_{ij} = \operatorname{sgn}(b_{ij}) \sqrt{|b_{ij} b_{ji}|}$, and $d_i b_{ij} = -d_j b_{ji}$ holds by assumption.  

- If $t_{ij} = -2 s_{ik} s_{kj}$, then
  \[
  s_{ik} s_{kj} = \sqrt{(-b_{ik} b_{ki})(-b_{kj} b_{jk})}
  = \sqrt{|(b_{ik} b_{kj})(-b_{jk} b_{ki})|},
  \]
  and using $d_i b_{ik} = -d_k b_{ki}$ and $d_k b_{kj} = -d_j b_{jk}$, we deduce
  \[
  d_i (b_{ik} b_{kj}) = (-d_k b_{ki}) b_{kj} = d_j (b_{jk} b_{ki})=-d_j(-b_{jk} b_{ki}),
  \]
  so the product $b_{ik} b_{kj}$ and $-b_{jk} b_{ki}$ are related by the same symmetrizer. Hence the term $s_{ik} s_{kj}$ arises from a pair of integers satisfying the skew-symmetrizer condition.

Therefore, the difference $T = S'' - \mathfrak{S}(B')$ is a sum of rank-two updates of the form handled in Lemma~\ref{lem:det4}. Applying Lemma~\ref{lem:det4} repeatedly (once for each nonzero off-diagonal entry of $T$), we conclude that
\[
\det(S'') \equiv \det(\mathfrak{S}(B')) \pmod{4}.
\]

Finally, since $\det(J + E) = \det(J) = -1$, we have
\[
\det(S'') = \det(J + E)^2 \det(\mathfrak{S}(B)) = \det(\mathfrak{S}(B)).
\]

Combining the above, we obtain
\[
\det(\mathfrak{S}(B)) = \det(S'') \equiv \det(\mathfrak{S}(B')) \pmod{4},
\]
which proves \eqref{eq:detmod2}. The proof is complete.
\end{proof}

Analogues of Lemmas~\ref{lem:integerdet} and~\ref{lem:det4} remain valid in this more general setting.

\begin{remark}\label{rmk:mod2}
  Let $S=(s_{ij}) \in M_n(\mathbb{R})$ be a real $n\times n$ matrix such that each entry $s_{ij} $  can be written as   
$$s_{ij} = \sum_{t \in T_{ij}} \epsilon(t) \sqrt{\,|b^t_{ij} b^t_{ji}|\,},$$
where $T_{ij}$ is a finite index set and  $\epsilon(t), b^t_{ij}, b^t_{ji} \in \mathbb{Z}$ for all $t\in T_{ij}$. Suppose further that there exist positive integers  $ d_1, \dots, d_n \in \mathbb{Z}_{>0} $  satisfying 
$$d_i b^t_{ij} = d_j b^t_{ji}
$$
for all $i,j\in [n]$ and all $t\in T_{ij}$. Then $\det(S)\in \mathbb{Z}$.

Moreover, for any fixed $i,j\in [n]$, any integers $b'_{ij}, b'_{ji} \in \mathbb{Z}$ with $d_ib'_{ij}=d_j b'_{ji} $, and any $ \epsilon \in \mathbb{Z}$, we have
$$    \det\bigl( S - 2\epsilon \sqrt{|b'_{ij} b'_{ji}|} \, \varepsilon_{ij} \bigr)
    \equiv \det(S) \pmod{2}.
$$
\end{remark}

\begin{remark}\label{rem:binary}
For any fixed $n$, Casals~\cite{Casals} showed that in the skew-symmetric case, the $\delta$-invariant takes only two possible values; more precisely,
$$
\delta(B) \in 
\begin{cases}
\{0,2\}, & \text{if } n \text{ is odd}, \\
\{0,1\}, & \text{if } n \equiv 0 \pmod{4}, \\
\{0,3\}, & \text{if } n \equiv 2 \pmod{4}.
\end{cases}
$$

In fact, as observed in Remark~ \ref{rmk:mod2}, the same dichotomy holds more generally for skew-symmetrizable matrices when $n$ is odd.

We also conjecture that for any fixed $n$ and any choice of skew-symmetrizer $d_1, \dots, d_n$, the invariant $\delta$ is binary. This is supported by several explicit computations.
\end{remark}
%
%

\printbibliography

@article {S2,
    AUTHOR = {Seven, Ahmet I.},
     TITLE = {Cluster algebras and symmetric matrices},
   JOURNAL = {Proc. Amer. Math. Soc.},
  FJOURNAL = {Proceedings of the American Mathematical Society},
    VOLUME = {143},
      YEAR = {2015},
    NUMBER = {2},
     PAGES = {469--478},
}

@article {S3,
    AUTHOR = {Seven, Ahmet I. and \"Unal, \.Ibrahim},
     TITLE = {Congruence invariants of matrix mutation},
   JOURNAL = {J. Pure Appl. Algebra},
  FJOURNAL = {Journal of Pure and Applied Algebra},
    VOLUME = {229},
      YEAR = {2025},
    NUMBER = {3},
     PAGES = {Paper No. 107920, 14},
}

@article {M,
    AUTHOR = {Muller, Greg},
     TITLE = {The existence of a maximal green sequence is not invariant
              under quiver mutation},
   JOURNAL = {Electron. J. Combin.},
  FJOURNAL = {Electronic Journal of Combinatorics},
    VOLUME = {23},
      YEAR = {2016},
    NUMBER = {2},
     PAGES = {Paper 2.47, 23},
}

@article {S1,
    AUTHOR = {Seven, Ahmet I.},
     TITLE = {Mutation classes of skew-symmetrizable {$3\times 3$} matrices},
   JOURNAL = {Proc. Amer. Math. Soc.},
  FJOURNAL = {Proceedings of the American Mathematical Society},
    VOLUME = {141},
      YEAR = {2013},
    NUMBER = {5},
     PAGES = {1493--1504},   
   }

@article {ABBS,
    AUTHOR = {Assem, Ibrahim and Blais, Martin and Br\"ustle, Thomas and
              Samson, Audrey},
     TITLE = {Mutation classes of skew-symmetric {$3\times 3$}-matrices},
   JOURNAL = {Comm. Algebra},
  FJOURNAL = {Communications in Algebra},
    VOLUME = {36},
      YEAR = {2008},
    NUMBER = {4},
     PAGES = {1209--1220},
}

@article{Casals,
  author = {Casals, Roger},
  title  = {A binary invariant of matrix mutation},
  journal = {J. Comb. Algebra},
  year   = {2024},
  doi    = {10.4171/JCA/99}
}

@article{FWZ,
  author = {Fomin, Sergey and Williams, Lauren and Zelevinsky, Andrei},
  title = {Introduction to cluster algebras. Chapters 1--3},
  journal = {arXiv:1608.05735},
  year = {2016},
}

@article{FN,
  author = {Fomin, Sergey and Neville, Scott},
  title = {Cyclically ordered quivers},
  journal = {arXiv:2406.03604},
  year = {2024},
}

@article{F,
  author = {Fomin, Sergey},
  title = {Quiver mutations},
  journal = {Open Problems in Algebraic Combinatorics, slides at OPAC 2022 website},
 year = {2022},
}

@article{FN1,
  author = {Fomin, Sergey and Neville, Scott},
  title = {Long mutation cycles},
  journal = {arXiv:2304.11505},
  year = {2023},
}

@article{N,
  author = {Neville, Scott},
  title = {Mutation-acyclic quivers are totally proper},
  journal = {arXiv:2409.17832},
  year = {2024},
}

@article{FZ,
    author = {Fomin, Sergey and Zelevinsky, Andrei},
    title = {Cluster algebras. {I}. {F}oundations},
    journal = {J. Amer. Math. Soc.},
    fjournal = {Journal of the American Mathematical Society},
    volume = {15},
    year = {2002},
    number = {2},
    pages = {497--529},
}

@article{BBH,
    author = {Beineke, Andre and Br\"ustle, Thomas and Hille, Lutz},
    title = {Cluster-cyclic quivers with three vertices and the {M}arkov equation},
    note = {With an appendix by Otto Kerner},
    journal = {Algebr. Represent. Theory},
    fjournal = {Algebras and Representation Theory},
    volume = {14},
    year = {2011},
    number = {1},
    pages = {97--112},
}

@article{E,
  author = {Ervin, Tucker},
  title = {Unrestricted red size and sign-coherence},
  journal = {arXiv:2401.14958},
  year = {2024},
}

@article{E2,
  author = {Ervin, Tucker},
  title = {New hereditary and mutation-invariant properties arising from forks},
  journal = {The Electronic Journal of Combinatorics},
    volume = {31},
    year = {2024},
    number = {1},
    pages = {16},
 }

\end{document}